\documentclass[11pt,a4paper]{article}
\usepackage[utf8]{inputenc}
\usepackage[english]{babel}
\usepackage[T1]{fontenc}
\usepackage{lmodern}
\usepackage{amsmath}
\usepackage{accents}
\usepackage{amsfonts}
\usepackage{amssymb}
\usepackage{amsthm}
\usepackage{amscd}
\usepackage{url}
\usepackage[shortlabels]{enumitem}
\usepackage{hhline}
\usepackage{dsfont}
\usepackage{array}

\usepackage{caption}
\usepackage{graphicx}
\usepackage{subcaption}

\usepackage{geometry}
\usepackage{tikz}
\usetikzlibrary{decorations.pathreplacing}
\usetikzlibrary{patterns}
\usepackage{xcolor}
\usetikzlibrary{arrows}
\usepackage{mdframed}
\usepackage{caption}
\usepackage[titletoc]{appendix}
\usepackage{csquotes}
\usepackage{hyperref}
\usepackage{float}
\usepackage{bm}
\usepackage{microtype}
\usepackage{graphicx}

\usepackage{wrapfig}

\mdfsetup{nobreak=true}

\AtBeginDocument{
  \label{CorterrectFirstPageLabel}
  
}

\usepackage[
style=alphabetic,
    sorting=nyt,   
    sortcites      
]{biblatex}
\newcolumntype{L}{>{$}l<{$}}

\newcommand\HH{\mathcal{H}}
\newcommand\J{\mathcal{J}}

\newcommand\R{\mathbb{R}}

\newcommand\x{\underline{x}}

\newcommand\y{\underline{y}}

\newcommand\Dd{\mathcal{D}}

\newcommand\E{\mathbb{E}}

\DeclareMathOperator\Var{Var}

\newcommand\dd{\mathrm{d}}
\newcommand\CC{\mathcal{C}}

\newcommand\I{\mathcal{I}}

\newcommand\nut{\tilde{\nu}}

\newcommand\Pt{\mathcal{P}}

\newcommand{\enstq}[2]{\left\{#1~\middle|~#2\right\}}

\DeclareMathOperator\Vol{Vol}

\newcommand\one{\mathds{1}}

\newcommand\numberthis{\addtocounter{equation}{1}\tag{\theequation}}
\newcommand\jump{\par\medskip}
\newcommand\quand{\quad\text{and}\quad}

\usepackage{bbm}
\usepackage{xcolor}

\usepackage{mathtools} \mathtoolsset{showonlyrefs,showmanualtags}
\numberwithin{equation}{section}

\newcommand{\be}{\begin{equation}}
\newcommand{\ee}{\end{equation}}
\newcommand{\bega}{\begin{equation}\begin{aligned}}
\newcommand{\eega}{\end{aligned}\end{equation}}

\makeatletter
\newcommand{\tpitchfork}{%
  \raise-0.1ex\vbox{
    \baselineskip\z@skip
    \lineskip-.52ex
    \lineskiplimit\maxdimen
    \m@th
    \ialign{##\crcr\hidewidth\smash{$-$}\hidewidth\crcr$\pitchfork$\crcr}
  }%
}

\theoremstyle{plain}
\newtheorem{theo}{Theorem}[section]
\newenvironment{theorem}%
  {\begin{mdframed}[backgroundcolor=white]\begin{theo}}%
  {\end{theo}\par\vspace{0.1cm}\end{mdframed}}
  
\theoremstyle{plain}
\newtheorem{coro}[theo]{Corollary}
\newenvironment{corollary}%
  {\begin{mdframed}[backgroundcolor=white]\begin{coro}}%
  {\end{coro}\par\smallskip\end{mdframed}}
  
\theoremstyle{plain}
\newtheorem{lemm}[theo]{Lemma}
\newenvironment{lemma}%
  {\begin{mdframed}[backgroundcolor=white]\begin{lemm}}%
  {\end{lemm}\par\vspace{0.cm}\end{mdframed}}

\theoremstyle{plain}
\newtheorem{prop}[theo]{Proposition}
\newenvironment{proposition}%
  {\begin{mdframed}[backgroundcolor=white]\begin{prop}}%
  {\end{prop}\par\vspace{0.cm}\end{mdframed}}
  
\theoremstyle{definition}
\newtheorem{defn}[theo]{Definition}
  {\begin{mdframed}[backgroundcolor=white]\begin{defn}}%
  {\end{defn}\par\vspace{0.cm}\end{mdframed}}

\theoremstyle{definition}
\newtheorem{remark}[theo]{Remark}

\makeatletter
\def\blfootnote{\gdef\@thefnmark{}\@footnotetext}
\makeatother
  
\makeatletter
\renewcommand*\env@matrix[1][*\c@MaxMatrixCols c]{%
  \hskip -\arraycolsep
  \let\@ifnextchar\new@ifnextchar
  \array{#1}}
\makeatother

\begin{document}

\title{{\bf Tightness of Stationary Nodal Measures}}
\author{Louis Gass and Giovanni Peccati}
\maketitle
\blfootnote{University of Luxembourg, Department of Mathematics}
\blfootnote{Research supported by the Luxembourg National Research Fund (Grant {\bf O24/18972745/GFRF})}
\blfootnote{Email: louis.gass(at)uni.lu,\quad giovanni.peccati(at)uni.lu}

\begin{abstract}
We study the \emph{rescaled nodal volume field} $\xi_R$ associated with a smooth, stationary Gaussian field on $[0,R]^d$, whose covariance satisfies adequate integrability conditions. Our main theorem shows that, as $R \to \infty$, the process $\xi_R$ converges in distribution, in an appropriate space of {c\`adl\`ag} mappings, to a standard Brownian sheet. The proof relies on a recent finite-dimensional CLT by Ancona {\it et al.} (2025), as well as on a multidimensional Kolmogorov--Chentsov criterion for tightness due to Bickel and Wichura (1971). The application of the latter requires new moment estimates that are of independent interest. Our results stand in sharp contrast with \emph{Berry's random wave model}, where the required integrability conditions fail and the question of tightness remains open.\jump

{\bf AMS 2010:} 60G60, 60F05, 34L20, 33C10.

\end{abstract}

%


\section{Introduction}
\subsection{Overview and motivation}\label{ss:overview}

This paper shows that the finite-dimensional central limit theorems (CLTs) 
for the nodal measure $\nu$ of a smooth stationary Gaussian field 
$f$ on $\mathbb{R}^d$, recently established in \cite{Gas25, Anc24} via chaotic and cumulant 
methods, can be extended to a full functional limit theorem. 
Our main result (Theorem~\ref{thm1}) proves that the partition function 
associated with a suitably rescaled version of the nodal measure $\nu$ converges in 
distribution to a Brownian sheet in the Skorokhod space $\mathcal{D}[0,1]^d$ 
of c\`adl\`ag mappings on the unit cube; see \cite{neuhaus} and Section \ref{ss:notation} below. 
Our key technical tool is a tightness criterion due to Bickel and Wichura 
\cite{Bic71}, which we implement through novel moment estimates of 
independent interest (see Theorem~\ref{thm2}).

As argued for instance in \cite[Appendix A]{NoPeVidotto} or \cite[Appendix A.4]{Smutek} (see also \cite{IvanovUkra}), this mode of 
convergence is strictly stronger than the more common {\it convergence 
to white noise} results appearing in the literature, see e.g.\ 
\cite{Smutek, AnconaLetendreKostlan, NoPeVidotto} and the discussion contained in \cite[Section 1.5]{Wigman2010}. 
It also stands in sharp contrast with situations where the limiting 
Gaussian random field is non-separable --- so that finite-dimensional 
convergence cannot be lifted to a functional limit, cf.\ 
\cite{SodinBuckley, SodinTsirelson} and the discussion in 
\cite{NoPeVidotto} --- as well as with the case in which the underlying 
Gaussian field has non-integrable covariance. 
In the latter regime, the proof of tightness is out of reach both for 
moment-based approaches and for techniques relying on Wiener chaos 
expansions \cite{NoPeVidotto}. 

As an illustration of the power of functional convergence, 
Corollary~\ref{Cor1} combines our main result with the classical 
computations of \cite{YehWiener} to derive new limit theorems for certain 
{\it non-local} functionals of nodal measures on $\mathbb{R}^2$. These objects are akin to 
the notion of {\it overcrowding}, as investigated for instance in 
\cite{Priya, overcrowding1, overcrowding2}, and appear to be out of reach of purely 
finite-dimensional arguments.\jump

Our results form the latest instalment in a growing line of work that uses tools from Gaussian analysis to derive asymptotic laws for local geometric functionals of Gaussian fields. We refer the reader to \cite{Kleon, NouPecRossi, MaRoPeccati, DaNouPecRossi, DierickxNouPecRos, AnconaLetendreKostlan, Gas25, PeccatiVidotto} for several distinguished contributions in this direction, and to \cite{WigmanSurvey} for a recent survey of the area.

\smallskip

\subsection{Notation}\label{ss:notation}

From now on, all random elements are assumed to be defined on a common probability space $(\Omega, \mathcal{F}, \mathbb{P})$, with $\mathbb{E}$ to indicating expectation with respect to $\mathbb{P}$. The symbol $\stackrel{\rm law}{\longrightarrow}$ is used to denote convergence in distribution of random variables and random vectors, whereas $\mathcal{N}(\mu, \sigma^2)$ stands for the one-dimensional Gaussian distribution with mean $\mu$ and variance $\sigma^2$. For $Z\sim\mathcal{N}(0,1)$, we write
\begin{equation}\label{e:Phi}
\Phi(z) := \mathbb{P}(Z\leq z), \quad z\in \R.
\end{equation}

\smallskip 

Given an integer $d\geq 1$, we consider the unit (hyper)cube $[0,1]^d$ and define $\mathcal{D}[0,1]^d = \mathcal{D}_d$ to be the class of {\it multiparameter (generalized) c\`adl\`ag functions} on $[0,1]^d$, as defined in \cite[Def. 1.1, p. 1286]{neuhaus}. We endow the space $\mathcal{D}[0,1]^d$ with the $\sigma$-field generated by coordinate projections, as well as with the Skorohod topology described e.g. in \cite[p.1289]{neuhaus}. We recall that this topology is generated by a distance, making $\mathcal{D}[0,1]^d$ a separable metric space. We also define $C([0,1]^d,\R)$ to be the subset of $\mathcal{D}[0,1]^d$ composed of continuous mappings. We note that, if $h:[0,1]^d\to \R$ takes the form $h(t):=\nu([0,t_1]\times\cdots\times [0,t_d])$ for some finite measure $\nu$ on $[0,1]^d$, then an application of the dominated convergence theorem implies that $h \in \mathcal{D}[0,1]^d$. Given a collection $\{X_n,X\} $ of $\mathcal{D}[0,1]^d$-valued random elements, we say that $X_n$ {\it converges in distribution} to $X$ if
$
\lim_n \mathbb{E}[\varphi(X_n)] =\mathbb{E}[\varphi(X)], 
$
for all $\varphi : \mathcal{D}[0,1]^d \to \R$ continuous and bounded. 
\smallskip

Given a locally finite measure $\nu$ on $\R^d$ and a compactly supported bounded test function $\phi : \R^d \to \R$, we define
\begin{equation}\label{e:action}
\langle \nu, \phi\rangle := \int_{\R^d} \phi(x) \, \nu(dx).
\end{equation}

\smallskip

For every $d\geq 1$, we define the {\it $d$-parameter standard Brownian sheet} to be the centered Gaussian field
\begin{equation}\label{e:bsheet}
W = \{W(t) : t = (t_1,...,t_d)\in [0,1]^d\}
\end{equation}
such that
\begin{equation}\label{e:cbsheet}
\mathbb{E}[W(s) W(t)] = \prod_{\ell=1}^d \min(s_\ell , t_\ell),\quad s,t \in [0,1]^d.
\end{equation}
It is well known that the paths of $W$ are $\mathbb{P}$-a.s. continuous, so that $W$ can be regarded as a $\mathcal{D}[0,1]^d$-valued random element. Plainly, for $d=1$ the process $W$ is simply a standard Brownian motion on $[0,1]$. Fig. \ref{fig:BF_combined} displays a visualization of $W$ in the case $d=2$.

\smallskip

\begin{figure}[h!]
    \centering
    \begin{minipage}{0.42\textwidth}
        \centering
        \includegraphics[width=\textwidth]{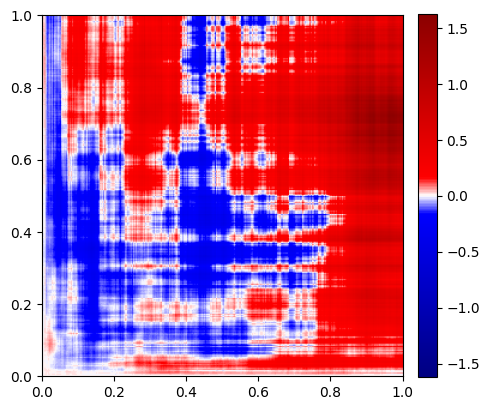}
    \end{minipage}
    \hfill
    \begin{minipage}{0.42\textwidth}
        \centering
        \includegraphics[width=\textwidth]{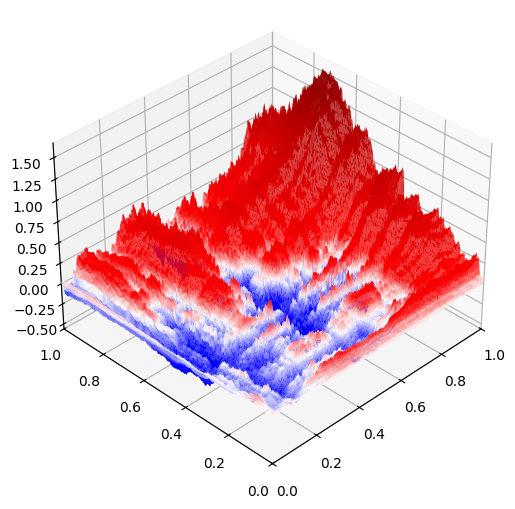}
    \end{minipage}

    \caption{\it Heatmap and 3D plot of a standard Brownian sheet $W$ on $[0,1]^2$.}
    \label{fig:BF_combined}
\end{figure}

\subsection{Setting}

Let $1\leq k\leq d$ be integers. In what follows, we consider a $\R^k$-valued centered stationary Gaussian field on $\R^d$, 
\begin{equation}\label{e:effe}
f = \big\{f(x) = (f^1(x),...,f^k(x)) : x\in \R^d\big\}.
\end{equation}

\smallskip

Given the random function in \eqref{e:effe}, we define the random nodal set
\[Z := \enstq{x\in \R^d}{f(x) = 0}.\]
It is easy to check that, under assumptions $(H1)$ or $(\widetilde{H}1)$ below, $Z$ is a.s.-$\mathbb{P}$ a smooth hypersurface of dimension $d-k$ in $\R^d$ (see e.g. \cite{Aza09, ATaylor}). We can then consider the associated {\it nodal measure} $\nu$ on $\R^d$ and its centered version $\nut$, defined for every measurable subset $A$ of $\R^d$ by
\begin{equation}\label{e:nueffe}
\nu(A) := \HH^{d-k}(Z\cap A)\quand \nut(A) := \nu(A)-\E[\nu(A)],
\end{equation}
where $\HH^{d-k}$ denotes the $d-k$ dimensional Hausdorff measure. We also associate with the random function $f$ the constant
\begin{equation}\label{e:gamma2}
\gamma_2 = \int_{\R^d} \rho_2(x)\dd x + \rho_1\one_{\{k=d\}}
\end{equation}
where $\rho_1$ and $\rho_2$ are the Kac densities defined in Equation \eqref{kac-density} below. Since the field is assumed to be invariant by translation, $\rho_1$ is a constant and $\rho_2$ can be seen as a function on $\R^d$.
We introduce the following set of assumptions:
\begin{itemize}
\item[$(H1)$] For any $\ell\in\{1,\ldots,8\}$, $m_1,\ldots,m_\ell$ such that $m_1+\ldots+m_\ell = 8$ and distinct points $x_1,\ldots,x_\ell\in\R^d$, the Gaussian vector $(\partial^{\alpha_i} f(x_i))_{1\leq i\leq \ell, \;\alpha_i\leq m_i}$ is non-degenerate. This assumption holds true as soon as the process $f$ is of class $\CC^8$ and the support of its spectral measure contains an open set.
\item[$(\widetilde{H}1)$] $f = \nabla h$, and $h$ is a stationary Gaussian process on $\R^d$ (so that $d=k$) and such that for any $\ell\in\{1,\ldots,9\}$, $m_1,\ldots,m_\ell$ such that $m_1+\ldots+m_\ell = 9$ and distinct points $x_1,\ldots,x_\ell\in\R^d$, the Gaussian vector $(\partial^{\alpha_i} f(x_i))_{1\leq i\leq \ell, \;\alpha_i\leq m_i}$ is non-degenerate. This assumption holds true as soon as the process $h$ is of class $\CC^9$ and the support of its spectral measure contains an open set.
\item[$(H2)$] There is a parameter $\varepsilon>0$ and a function $g\in L^2(\R_+)$ such that
\[\sup_{|u|,|v|\leq \varepsilon}\sup_{|\alpha + \beta|\leq 7} |\E[f^{(\alpha)}(x+u)f^{(\beta)}(y+v)]|\leq g(\|x-y\|).\]
This condition holds true as soon as the covariance function and its derivatives decay as a rate faster than $\|x\|^{-\alpha}$, with $\alpha>d/2$.
\item[$(H3)$] The constant $\gamma_2$ is positive. In the recent paper \cite{Gas25} (see Thm. 1.7) the positivity of the constant $\gamma_2$ is discussed: for instance, if $f$ is a collection of $k$ independent real random fields or if $f$ is a gradient field, then $\gamma_2$ is positive. In particular, $(H3)$ holds true if $k=1$, or if $(\widetilde{H}1)$ is satisfied.
\end{itemize}

\medskip

A prototypical example of a Gaussian field $f$ verifying assumptions $(H1)$, $(H2)$ and $(H3)$ for $k=1$ is the so-called {\it Bargmann-Fock field}, whose covariance is given by
\begin{equation}\label{e:BF}
\mathbb{E}[f(x)f(y)]= \exp\{ -\|x-y\|^2\}, \quad x,y\in \R^d,
\end{equation}
and which is also the field featured in the simulations of Figures \ref{fig:R} and \ref{fig:LL}. As discussed, e.g., in the survey \cite{Nalini}, such a field is the (large-degree) scaling limit of the {\it Kostlan polynomials} ensemble. A further example would be the so-called {\it large-band field}, whose spectral measure is the indicator of the unit ball in $\R^d$ and appears as the scaling limit of the Riemannian random wave model; see \cite{GassBer} for more details.

\jump

\subsection{Main results}
When $f$ satisfies assumption $(H1)$ (or $(\widetilde{H}1)$) and $(H2)$, 
it has been proved---by Wiener chaos methods in \cite[Thm 1.7]{Gas25}, 
and by cumulant methods in \cite[Thm. 1.9]{Anc24} under additional smoothness assumptions---that 
a suitable renormalization of $\nu(A)$ satisfies a central limit theorem as the set $A$ diverges to infinity, subject to a mild regularity condition on $A$. More precisely, for $R>0$ we introduce the normalized signed random measure $\nut_R$, whose action on a compactly supported bounded test function $\phi$ is defined by
\[
\langle \nut_R,\phi\rangle
:= \frac{\langle \nu,\phi\!\left(\frac{\cdot}{R}\right)\rangle - 
\E[\langle \nu,\phi\!\left(\frac{\cdot}{R}\right)\rangle]}
{\sqrt{\gamma_2}R^{d/2}},\numberthis\label{eq:04}
\]
where $\gamma_2$ is given in \eqref{e:gamma2} and the notation \eqref{e:action} and \eqref{e:nueffe} is in force. For every fixed compactly supported bounded test function $\phi$, the results in 
\cite[Thm. 1.7]{Gas25} yield the following central limit theorem under assumptions weaker than $(H1)$ or $(\widetilde{H1})$, $(H2)$ and $(H3)$
\[
\langle \nut_R,\phi\rangle 
\,\,{\overset{{\rm law}}{\longrightarrow}}\,\, 
\mathcal{N}(0,\|\phi\|_2),
\qquad R\to\infty.\numberthis\label{eq:01}
\]

\smallskip

As announced, our aim is to establish a functional counterpart of \eqref{eq:01}.  
For a vector $t\in\R_+^d$, write $[0,t]=[0,t_1]\times\cdots\times[0,t_d]$.  
The {\it nodal volume random field}
\[
\xi_R=\{\xi_R(t):t\in[0,1]^d\},
\]
corresponding to the partition function associated with the restriction of $\nut_R$ to $[0,1]^d$, is defined by
\[
\xi_R(t):=\nut_R([0,t])
= \langle \nut_R,\one_{[0,t]}\rangle,
\qquad t\in[0,1]^d. \numberthis\label{e:xi}
\]
By the Cramér--Wold device, the convergence in \eqref{eq:01} implies that the finite-dimensional 
distributions of $\xi_R$ converge to those of a standard Brownian sheet $W$ on $[0,1]^d$, 
as defined in \eqref{e:bsheet}--\eqref{e:cbsheet}. The following theorem---our main result---upgrades this finite-dimensional convergence to functional convergence.

\begin{theorem}
\label{thm1}
Let $f$ be a random process satisfying assumptions $(H1)$ (resp. $(\widetilde{H}1)$), $(H2)$ and $(H3)$. Then, as $R\to+\infty$, the random field $\xi_R$ defined in \eqref{e:xi} converges in distribution in $\Dd([0,1]^d)$ towards a $d$-dimensional standard Brownian sheet on $[0,1]^d$.
\end{theorem}

An illustration of Theorem \ref{thm1} in the case where $f$ is the Bargmann-Fock field defined in \eqref{e:BF}  is provided in Figure \ref{fig:R}. 

\begin{figure}[H]
\centering

\begin{subfigure}{0.32\textwidth}
    \centering
    \includegraphics[width=\linewidth]{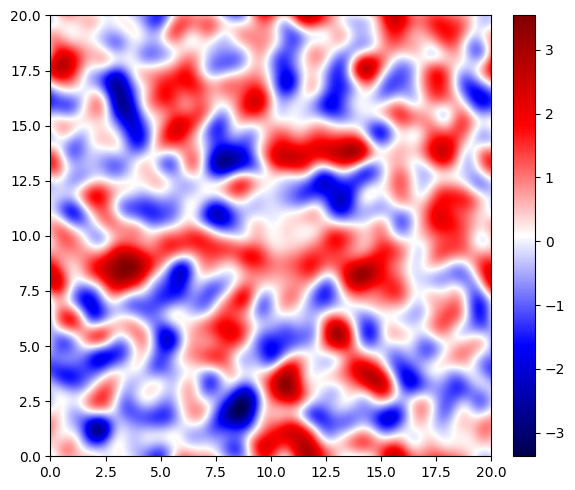}
    \caption*{(a)}
\end{subfigure}
\hfill
\begin{subfigure}{0.32\textwidth}
    \centering
    \includegraphics[width=\linewidth]{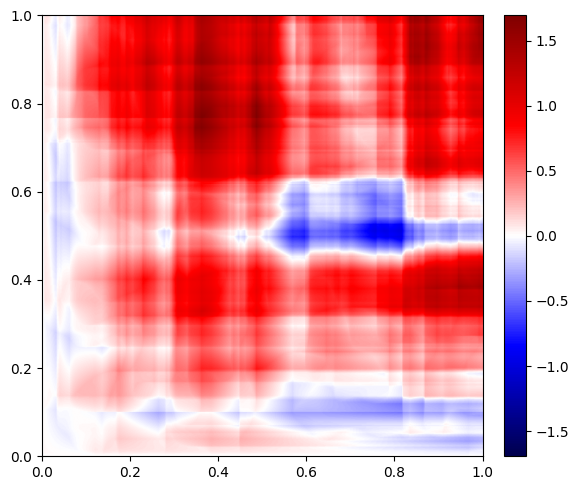}
    \caption*{(b)}
\end{subfigure}
\hfill
\begin{subfigure}{0.32\textwidth}
    \centering
    \includegraphics[width=\linewidth]{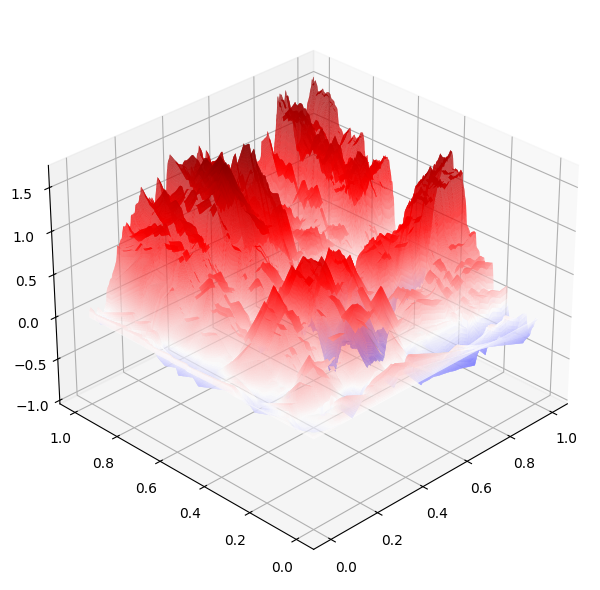}
    \caption*{(c)}
\end{subfigure}

\vspace{0.3cm}

\begin{subfigure}{0.32\textwidth}
    \centering
    \includegraphics[width=\linewidth]{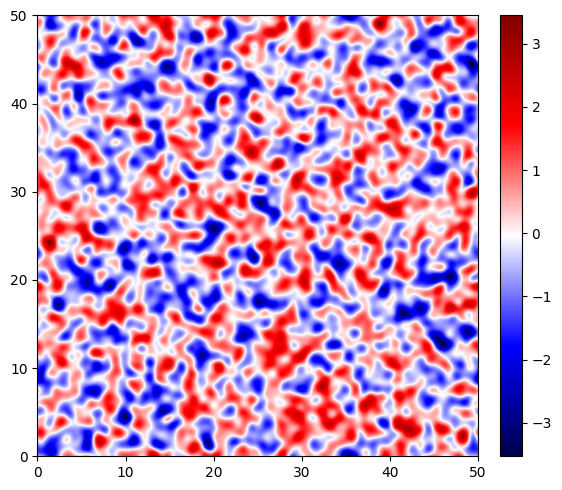}
    \caption*{(d)}
\end{subfigure}
\hfill
\begin{subfigure}{0.32\textwidth}
    \centering
    \includegraphics[width=\linewidth]{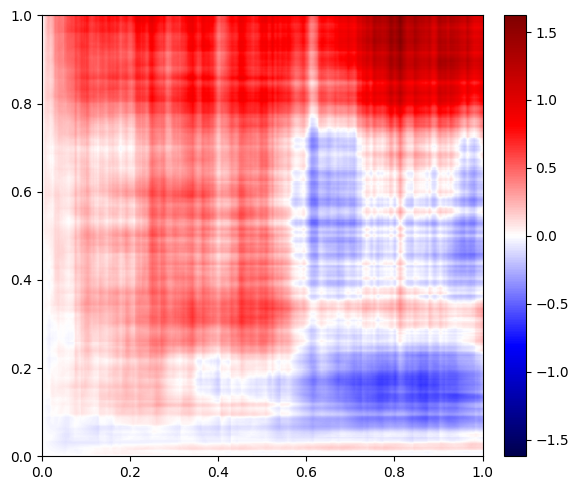}
    \caption*{(e)}
\end{subfigure}
\hfill
\begin{subfigure}{0.32\textwidth}
    \centering
    \includegraphics[width=\linewidth]{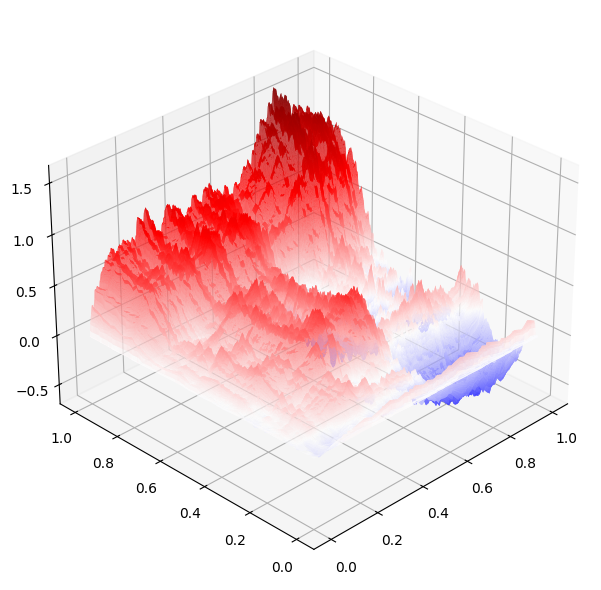}
    \caption*{(f)}
\end{subfigure}

\vspace{0.3cm}

\begin{subfigure}{0.32\textwidth}
    \centering
    \includegraphics[width=\linewidth]{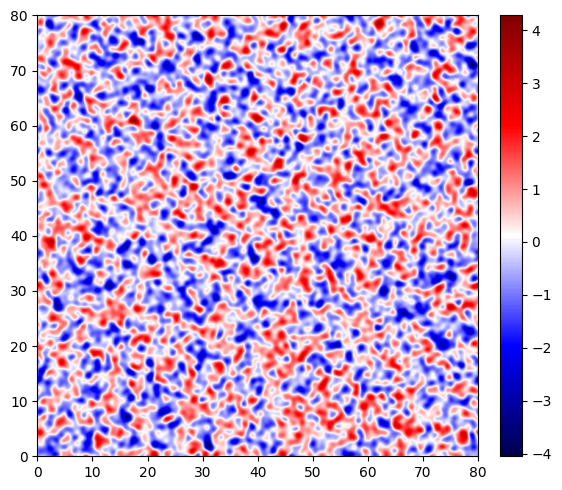}
    \caption*{(g)}
\end{subfigure}
\hfill
\begin{subfigure}{0.32\textwidth}
    \centering
    \includegraphics[width=\linewidth]{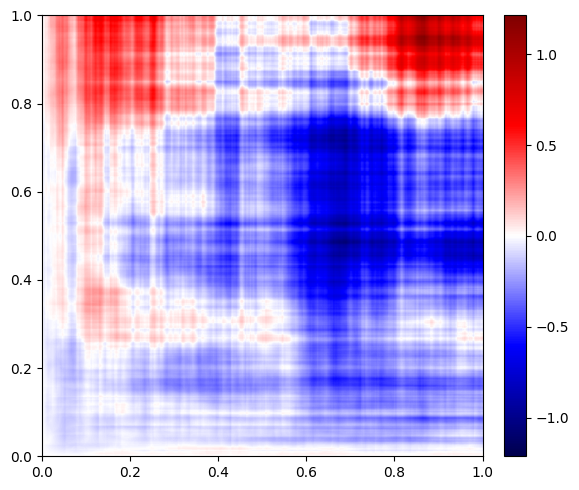}
    \caption*{(h)}
\end{subfigure}
\hfill
\begin{subfigure}{0.32\textwidth}
    \centering
    \includegraphics[width=\linewidth]{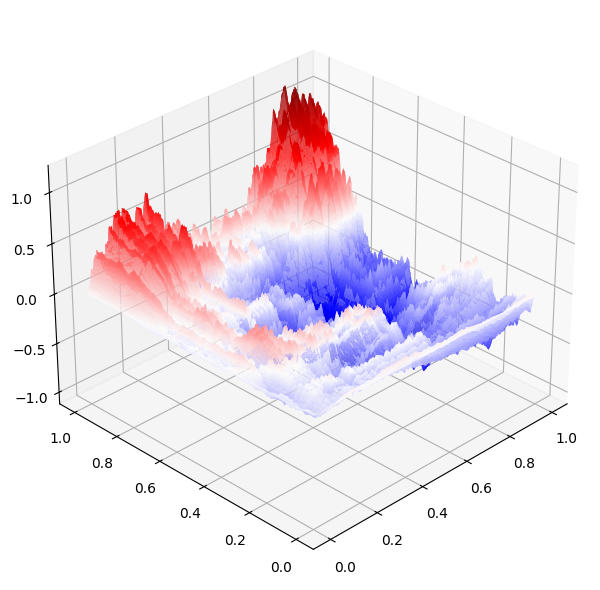}
    \caption*{(i)}
\end{subfigure}

\vspace{0.3cm}

\caption{\it Simulations of the Bargmann--Fock field on the square $[0,R]^2$, and of the associated nodal-length field $\xi_R$ on $[0,1]^2$. From left to right, the columns display the Bargmann--Fock field {\rm ((a), (d), (g))}, the nodal-length field $\xi_R$ shown as a heatmap {\rm ((b), (e), (h))}, and a three-dimensional plot of $\xi_R$ {\rm ((c), (f), (i))}. From top to bottom, the rows correspond to $R=20$, $R=50$, and $R=80$.}
\label{fig:R}
\end{figure}

\begin{remark}
Assumptions $(H1)$, $(\widetilde{H}1)$ and $(H2)$ can be relaxed by considering a sequence $\{f_n: n\geq 0\}$ converging in distribution to a stationary field $f$, as discussed in \cite{Anc24}. Moreover, when $k<d$, one can show that the mapping $t\mapsto \xi_R(t)$ is $\mathbb{P}$- a.s. continuous on $[0,1]^d$; see, e.g., \cite{Pec24}. Since the space $C([0,1]^d)$ endowed with the uniform topology is closed in the Skorokhod space $\Dd[0,1]^d$, one deduces the functional convergence established above also holds in $C([0,1]^d)$.
\end{remark}

\medskip 

Since $W$ is continuous a.s.-$\mathbb{P}$, one immediately deduces from 
Theorem \ref{thm1} that, for any compact set $L\subseteq[0,1]^d$,
\begin{equation}\label{e:maxlaw}
\sup_{t\in L}\, \xi_R(t) \stackrel{\rm law}{\longrightarrow} \sup_{t\in L}\, W(t).
\end{equation}
When $d=1$, it is well-known that $\sup_{t\in [0,1]} W(t) \stackrel{\rm law}{=}|G|$, where $G$ is a standard Gaussian random variable (see e.g. \cite[Theorem 2.21]{legall}). When $d\geq 2$, the exact distribution of the random variable 
$\sup_{t\in L} W(t)$ is known only in a handful of cases; for $d=2$, 
several of these are discussed in the classical reference \cite{YehWiener}, 
whose content is recalled below.

\begin{wrapfigure}[8]{r}{0.44\textwidth}
  \vspace{-\baselineskip}
  \centering
  \includegraphics[width=0.48\linewidth]{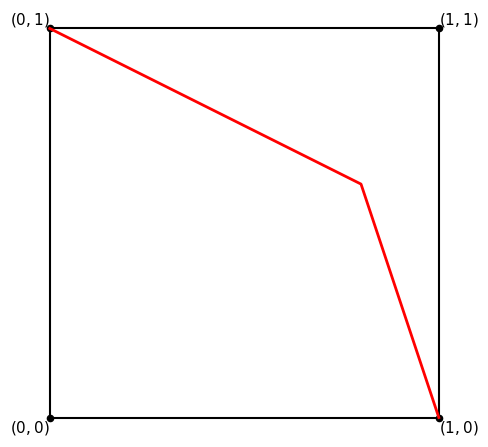}%
  \hfill
  \includegraphics[width=0.48\linewidth]{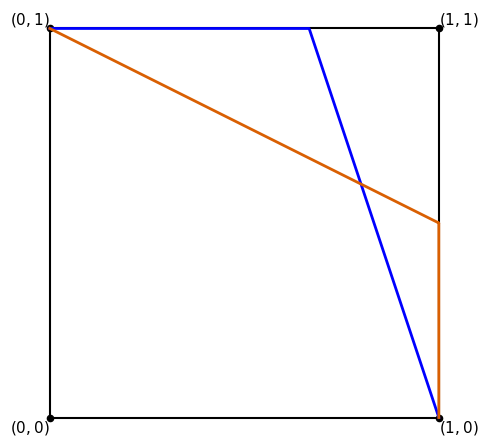}
  \caption{\it The set $L(a,b)$ for $a=2, \, b=3$ (red curve), 
  $a=\infty,\, b=3$ (blue curve), and $a=2,\, b=\infty$ (orange curve). The breaking point in each curve has coordinates $(k, 1-k/a)$.}
  \label{fig:L}
\end{wrapfigure}

\medskip

For every $a,b\in (1,+\infty]$, we introduce the notation
\[
k := \frac{a(b-1)}{ab-1}
\qquad\text{and}\qquad
c:= \frac{a(b-1)}{b(a-1)},
\]
with the convention
\[
(k,c)
=
\begin{cases}
\left(\dfrac{b-1}{b},\, \dfrac{b-1}{b}\right), & a = +\infty,\; b \text{ finite}, \\[6pt]
\left(1,\, \dfrac{a}{a-1}\right),              & b = +\infty,\; a \text{ finite}, \\[6pt]
(1,1),                                         & a = b = +\infty.
\end{cases}
\]
For $a,b$ as above, we define the polygonal curve $L(a,b)$ by
\[
L(a,b)
:= \{(x,y)\in[0,1]^2 : y= 1-x/a \text{ for } x\le k,\ 
                         y = b-bx \text{ for } x\ge k\},
\]
as illustrated in Figure \ref{fig:L}. We observe that 
$L(+\infty,+\infty)$ coincides with union of the upper and right edges of the unit square.
Then, in \cite[Theorem 2]{YehWiener} it is shown that, for every $\lambda\ge 0$,
\begin{eqnarray}\label{e:repf}
 &&\mathbb{P}\!\left[\sup_{t\in L(a,b)}\, W(t)\leq \lambda\right] \\ \notag
 &&= \Phi\!\big(\lambda(a+c)/(ac^{1/2})\big)
    - e^{-2\lambda^2/a}\,\Phi\!\big(\lambda(c-a)/(ac^{1/2})\big) \\ \notag 
 &&\quad -e^{-2\lambda^2/b}\,\Phi\!\big(\lambda(1-bc)/(bc^{1/2})\big)
    +e^{-2\lambda^2(a^{-1}+b^{-1}-2)}\,
     \Phi\!\big(\lambda c^{-1/2}(b^{-1}-c-2 )\big)
:= H(a,b,\lambda),
\end{eqnarray}
where we have adopted the notation \eqref{e:Phi}.
From Theorem \ref{thm1} one immediately deduces the following:

\smallskip
\par\noindent
\begin{minipage}{\textwidth}
\begin{corollary}\label{Cor1}
Consider the case $d=2$. Then, under the assumptions of Theorem \ref{thm1},
for every $a,b>1$ as above,
\[
\mathbb{P}\!\left[\sup_{t\in L(a,b)}\, \xi_R(t)\leq \lambda\right] 
\;\longrightarrow\; H(a,b,\lambda),
\qquad \lambda\ge 0.
\]
When $a=b=+\infty$, this implies in particular that
\[
\mathbb{P}\!\left[\sup_{t\in \partial[0,1]^2}\, \xi_R(t)\leq \lambda\right]
\;\longrightarrow\; H(+\infty,+\infty,\lambda)
= 1-3\Phi(-\lambda)+e^{4\lambda^2}\Phi(-3\lambda).
\]
\end{corollary}
\end{minipage}

\medskip

As already noted in Section \ref{ss:overview}, any estimate on the quantity
\[
1-  \mathbb{P}\!\left[\sup_{t\in L}\, \xi_R(t)\leq \lambda\right]
= \mathbb{P}\!\left[\sup_{t\in L}\, \xi_R(t)> \lambda\right]
= \mathbb{P}\!\left[\sup_{t\in L}\, \{\nu_R([0,t]) 
       - \mathbb{E}[\nu_R([0,t])] \} > \gamma_2^{1/2} R \lambda\right],
\]
can be interpreted as the assessment of a {\it non-local overcrowding event} in the spirit of 
\cite{Priya, overcrowding1, overcrowding2}, as it quantifies the probability 
that there exists $t\in L$ such that $\nu_R([0,t])$ exceeds its expectation 
by a margin of order~$R$.  
In particular, as $\lambda\to\infty$,
\[
\mathbb{P}\!\left[\sup_{t\in \partial[0,1]^2} W(t)> \lambda\right]
= 1 - H(+\infty,+\infty,\lambda)
\lesssim \frac{e^{-\lambda^2/2}}{\lambda}.
\]

\begin{wrapfigure}[11]{r}{0.44\textwidth}
  \vspace{-0.3\baselineskip}
  \centering
  \includegraphics[width=0.48\linewidth]{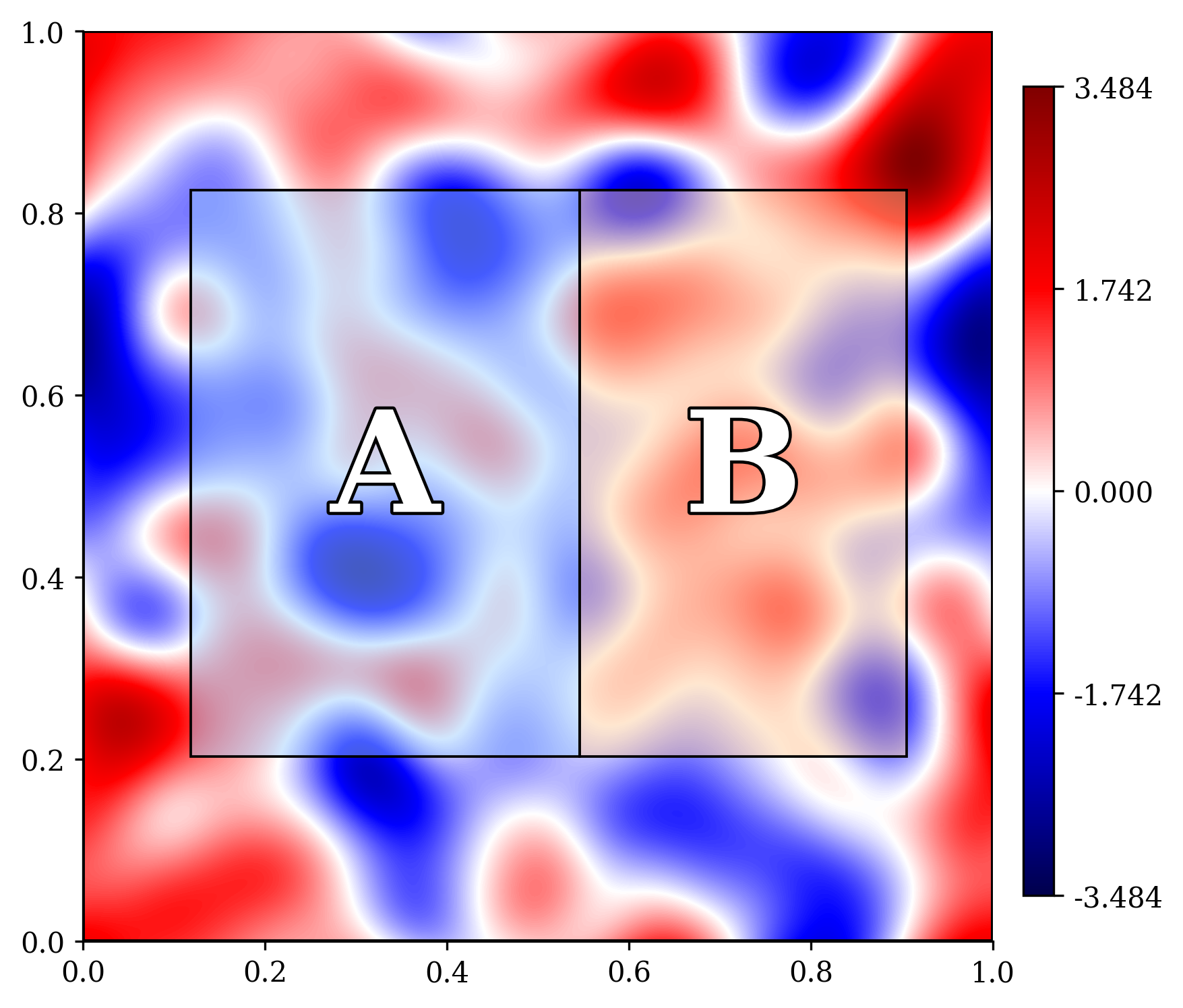}%
  \hfill
  \includegraphics[width=0.48\linewidth]{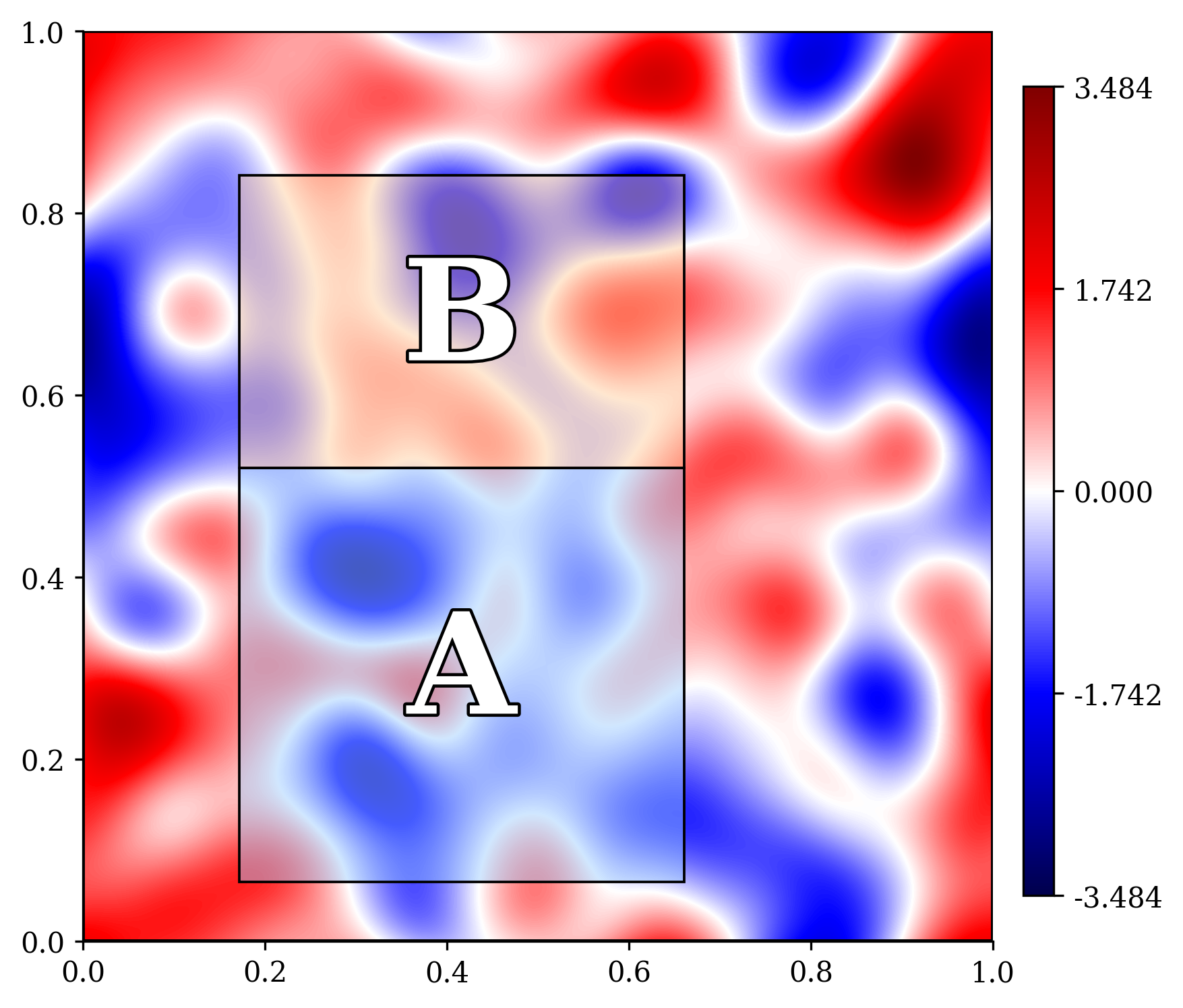}
  \caption{\it Two pairs of adjacent rectangles $(A,B)$ contained in $[0,1]^2$. In the background, a realization of $x\mapsto f(Rx)$, where $f$ is the Bargmann-Fock field and $R=10$.}
  \label{fig:LL}
\end{wrapfigure}

Since the convergence of the finite-dimensional distributions of the fields $\xi_R$, $R>0$, follows from \eqref{eq:01}, Prokhorov’s theorem (see, e.g., \cite[Theorem~15.1]{billingsley}) implies that, in order to prove Theorem~\ref{thm1}, it suffices to establish tightness for the family of laws associated with the class $\{\xi_R : R>0\}$. Such tightness can be deduced from the multidimensional Kolmogorov--Chentsov criterion proved in \cite{Bic71}.  Recall that, according to the nomenclature of \cite{Bic71}, two rectangles $A,B\subset[0,1]^d$ are \emph{adjacent} if their interiors are disjoint and their union is again a rectangle (see Figure \ref{fig:LL}). Also, for every rectangle $A$, the quantity $\nu_R(A)$ coincides precisely with the \emph{increment} of the process $\xi_R$ over $A$ in the sense of \cite[p.~3]{Bic71}.
 We fix a positive exponent $\alpha$ such that 
\[\alpha < \left\lbrace\begin{array}{llll} 
\frac{1}{2}\vphantom{\int_{\int_{\int}}} \qquad\quad\;\text{if}\;d=1,\\
\frac{1}{d^2-d+1} \quad\text{if}\;k=d>1\\
\frac{d-k}{2d} \qquad\;\;\text{if}\;k<d\vphantom{\int^{\int^\int}}.
\end{array}\right.\]

The following statement is the main technical achievement of our work.
\begin{theorem}
\label{thm2}
Let $f$ be a ranom field on $\R^d$ satisfying assumptions $(H1)$ (resp. $(\widetilde{H}1)$), $(H2)$ and $(H3)$. Then, there exists a constant $C$ depending only on the process $f$ such that, for $R\geq 1$ and $A,B$ two adjacent rectangles in $[0,1]^d$,
\[\E[|\nut_R(A)|^{3/2}|\nut_R(B)|^{3/2}]\leq C \Vol(A\cup B)^{1+\alpha}.\]
In particular, the laws associated with the class $\{ \xi_R: R>0\}$ are tight in $\Dd[0,1]^d$.
\end{theorem}
The proof Theorem \ref{thm2} relies on several fine estimates of the moments of order $2$ and $4$ of the random measure $\nut_R$, stated below in Proposition \ref{prop1} and Proposition \ref{prop2}. The main difficulty resides in having moment estimates for the nodal measure that are valid for rectangles with arbitrary size and flatness. Indeed, according to \cite{Gas23}, although the moments of these quantities are finite, no sufficiently strong estimates are available, allowing one to directly apply the results in \cite{Iva82} (except in the case $k=d=1$).
\section{Proof}
\subsection{Kac--Rice formula and cumulants}
We first recall the Kac--Rice formula, in a version whose proof can be found in \cite[Thm. 6.2]{Aza09}. There will be a distinction to be made between the case $k<d$ and $k=d$. Indeed, in the latter case, the Kac--Rice formula yields the factorial moments for the number of zeros, rather than the standard moments. This will change slightly the underlying combinatorial arguments, while keeping the core strategy behind the proof unchanged. In the following, $p$ is a positive integer and $\Pt_p$ stands for the set of all partitions of the set $\{1,\ldots,p\}$. Let $\Delta$ be \textit{large diagonal} in $(\R^d)^p$, defined by
\[\Delta = \enstq{(x_1,\ldots,x_p)\in (\R^d)^p}{\exists i\neq j\;\text{such that}\; x_i=x_j}\]
We define the {\it Kac density} $\rho_p$ on $(\R^d)^p\setminus\Delta$, as
\[\rho_p(x_1,\ldots,x_p)= \dfrac{\E\left[\prod\limits_{a\in A}\sqrt{\det[(\nabla_{x_a}f)( \nabla_{x_a}f)^T]}\;\middle|\; f(x_a)=0, \, a\in A\right]}{\sqrt{\det\left[2\pi\Var((f(x_a))_{a\in A})\right]\vphantom{\int^a}}}.\numberthis\label{kac-density}\]
In the following, we fix two adjacent rectangles $A,B$ in $\R_+^d$. We can assume by stationarity that their union is the rectangle $[0,t]$, with $t = \R_+^d$. We can also assume without loss of generality that the coordinates of $t$ are ordered. Let $i\in \{0,\ldots,d\}$ be the integer such that
\[t_1\leq \ldots \leq t_i< 1\quand 1\leq t_{i+1}\leq\ldots\leq t_d.\]
Heuristically, one must think about $t_1,\ldots t_i$ as arbitrarily small coordinates, and $t_{i+1},\ldots t_d$ as arbitrarily large coordinates. The following statement combines the already mentioned Kac–Rice formula with the main results of \cite{Gas23,Anc23}, and shows that, under the assumptions of Theorem~\ref{thm1}, the relevant moments are finite.
\begin{theorem}[Kac--Rice]
\label{thm3} Let $f$ be a random field on $\R^d$ satisfying assumptions  $(H1)$ (resp. $(\widetilde{H}1)$). Let $p$ be a positive integer and $A,B$ be two measurable subsets of $\R^d$ with disjoint interiors. If $k<d$ then
\[E[\nu(A)^p\nu(B)^p] = \int_{A^p\times B^p}\rho_{2p}(\x)\dd \x <+\infty,\]
and if $d=k$ then
\[E[\nu(A)^p\nu(B)^p] = \sum_{\I,\J\in \Pt_p}\int_{A^{\I}}\int_{B^{\J}}\rho_{|\I|+|\J|}(\x_{\I},\y_{\J})\dd \x_{\I}\dd \y_{\J}<+\infty.\]
\end{theorem}


We now introduce the (factorial) cumulant Kac-density, defined by 
\begin{equation}\label{e:Fp}
F_p(x_1,\ldots,x_p) = \sum_{\I\in \Pt_p}(|\I|-1)!(-1)^{|\I|-1}\prod_{I\in\I}\rho_{|I|}((x_i)_{i\in I}),
\end{equation}
so that
\[\rho_p(x_1,\ldots,x_p) = \sum_{\I\in \Pt_p}\prod_{I\in\I}F_{|I|}((x_i)_{i\in I}).\]
As announced, to compute the centered moments in terms of cumulant Kac density, there is a  distinction to be made between the case $k=d$ and $k<d$, due to the different combinatorial nature underlying the two settings. One could write a general formula valid for every $p$, based for instance on \cite[Prop. 3.5]{Gas21t}, but for our application we only need to specify it for moments of order $2$ and $4$.
\begin{lemma}
\label{lem1}
Let $f$ be a random field on $\R^d$ satisfying assumptions  $(H1)$ (resp. $(\widetilde{H}1)$). If $k<d$, then
\[E[\nut(A)^2] = \int_{A^2}F_2(\x)\dd \x,\]
\[E[\nut(A)\nut(B)] = \int_{A\times B} F_2(\x)\dd \x,\]
and
\begin{align*}
E[\nut(A)^2\nut(B)^2] = \int_{A^2\times B^2}F_4(\x)\dd x + 2\E[\nut(A)\nut(B)]^2 + \E[\nut(A)^2]\E[\nut(B)^2].
\end{align*}
\end{lemma}
\begin{lemma}
\label{lem2} Let the assumptions of Lemma \ref{lem1} prevail. If $k=d$, then
\[E[\nut(A)^2] = \int_{A^2}F_2(\x)\dd \x + \int_{A}F(x)\dd x,\]
\[E[\nut(A)\nut(B)] = \int_{A\times B} F_2(\x)\dd \x,\]
and
\begin{align*}
E[\nut(A)^2\nut(B)^2] &= \int_{A^2\times B^2}F_4(\x)\dd x + \int_{A\times B^2}F_3(\x)\dd x + \int_{A^2\times B}F_3(\x)\dd x + \int_{A\times B}F_2(\x)\dd x\\
&+\vphantom{\int^{\int}}\quad 2\E[\nut(A)\nut(B)]^2 + \E[\nut(A)^2]\E[\nut(B)^2].
\end{align*}
\end{lemma}

\smallskip

 In the next two sections, we establish the estimates that are necessary to prove our main results.

\subsection{Second moment estimate}
The following lemma concerns the behavior of the Kac density $F_2$ defined in \eqref{e:Fp}. We define the parameter
\[\beta  = \left\lbrace\begin{array}{llll}
0 \qquad\;\;\,\text{if}\; d=1,\\
d-2 \quad\text{if}\;k=d,\\
k \qquad\;\;\,\text{if}\;k<d.
\end{array}\right.\]
\begin{lemma}
\label{lem3}
Let $f$ be a random field on $\R^d$ satisfying assumptions  $(H1)$ (resp. $(\widetilde{H}1)$) and $(H2)$. Then there is a constant $C$, depending only on the process $f$, such that
\[|F_2(x,y)|\leq C\|x-y\|^{-\beta}\one_{\|x-y\|\leq 1} + Cg^2(\|x-y\|)\one_{\|x-y\|\geq 1}.\]
In particular, the function $F_2$ is integrable.
\end{lemma}
\begin{proof}
From the introduction of \cite{Gas23} or the ideas in \cite{Gas21b}, one can show by the method of divided differences that one has the near-diagonal bound
\[\forall \|x-y\|\leq 1,\quad \rho_2(x,y)\leq C\|x-y\|^{-\beta}.\]
Since $\rho_1$ is a constant function, the near-diagonal inequality with $\rho_2$ replaced by $F_2$ follows.\jump

If the Gaussian vectors $(f(x), \nabla_x f)$ and $(f(y),\nabla_y f)$ are independent, notice from Equation \ref{kac-density} that $\rho_2(x,y) = \rho_1(x)\rho_1(y)$. It implies from \eqref{e:Fp} that $F(x,y) = 0$. When the distance between $x$ and $y$ is large, Hypothesis $(H2)$ implies that these two Gaussian vectors are close to independence, and we expect $|F(x,y)|$ to be small. The techniques introduced in \cite{Gas21b}, see also \cite[Lemma 8.2]{Anc24}, prove the inequality
\[\forall \|x\|\geq 1,\quad F_2(x)\leq C g^2(\|x\|).\]
\end{proof} 
We are now ready to state the announced second-moment estimate. Let $\eta$ be a positive parameter such that 
\[\beta/(1-\eta)<d\numberthis\label{eq:07}.\]
\begin{proposition}
\label{prop1}
 Let the assumptions of Lemma \ref{lem3} prevail. Then there is a positive exponent $\eta$ and a constant $C$ depending only on $\eta$ and on the process $f$, such that
\[\E[|\nut(A)\nut(B)|]\leq C\left(\prod_{j=1}^d t_j\right)\min\left(1, \left(\prod_{j=1}^i t_j\right)^{\eta}\left(\prod_{j=i+1}^d t_j\right)\right).\]
\end{proposition}
\begin{proof}
On one hand, one has from Lemma \ref{lem3}.
\begin{align*}
E[|\nut(A)\nut(B)|]&\leq \E[\nu(A)\nu(B)] + 3\E[\nu(A)]\E[\nu(B)]\\
&\leq \int_{[0,t]^2}\rho_2(x-y)\dd x\dd y + 3C\left(\prod_{j=1}^d t_j\right)^2\\
&\leq C\left(\prod_{j=1}^d t_j\right)\int_{[-t,t]\cap B(0,1)}\|x\|^{-\beta}\dd x + C\left(\prod_{j=1}^d t_j\right)^2.
\end{align*}
Using Hölder inequality, one gets from the choice of $\eta$ in \eqref{eq:07} the bound
\begin{align*}
\int_{[-t,t]\cap B(0,1)}\|x\|^{-\beta}\dd x&\leq \Vol([-t,t]\cap B(0,1))^{\eta}\left(\int_{B(0,1)}\|x\|^{-\frac{\beta}{1-\eta}}\dd x\right)^{1-\eta}\\
&\leq C\left(\prod_{j=1}^i t_j\right)^{\eta},\numberthis\label{eq:03}
\end{align*}
which implies that
\[E[|\nut(A)\nut(B)|]\leq C\left(\prod_{j=1}^i t_j\right)^{1+\eta}\left(\prod_{j=i+1}^d t_j\right)^2.\]
On the other hand, one has again from Lemma \ref{lem3}, and Lemma \ref{lem2},
\begin{align*}
\E[|\nut(A)\nut(B)|]&\leq \sqrt{\E[\nut(A)^2]\E[\nut(B)^2]}\\
&\leq C\int_{[0,t]^2}|F_2(x,y)|\dd x\dd y + C\int_{[0,t]}\rho_1\dd x\\
&\leq \left(\prod_{j=1}^d t_j\right)\int_{[0,t]}|F_2(x)|\dd x + C\left(\prod_{j=1}^d t_j\right)\numberthis\label{eq:02}\\
&\leq C\left(\prod_{j=1}^d t_j\right),
\end{align*}
since the function $F_2$ is integrable. Gathering both estimates, the conclusion follows.
\end{proof}
\begin{remark}
\label{rem1}
When $k<d$, one can prove in a similar manner the stronger bound,
\[\E[|\nut(A)\nut(B)|]\leq C\left(\prod_{j=1}^i t_j\right)^{1+\eta}\left(\prod_{j=i+1}^d t_j\right),\]
since the right hand term $\int_{[0,1]}\rho_1\dd x$ disappear in \eqref{eq:02}, and following Equation \eqref{eq:03}, the estimate of the left hand term can be improved to
\[\int_{[0,t]}|F_2(x)|\dd x\leq \left(\prod_{j=1}^i t_j\right)^\eta.\] 
\end{remark}
\subsection{Fourth moment estimate}
It has been proved in the recent paper \cite{Anc24} that the cumulant Kac density enjoys nice integrability properties. In particular, under our set of assumptions, the results of \cite[Sec. 7.3]{Anc24} implies the following lemma.
\begin{lemma}
\label{lem4}
 Let the assumptions of Lemma \ref{lem3} prevail. Then there is a constant $C$ depending only on the process $f$ such that
\[\int_{[0,t]^p} |F_p(x_1,\ldots,x_p)|\dd x_1,\ldots\dd x_p\leq C\left(\prod_{j=1}^i t_j\right)\left(\prod_{j=i+1}^d t_j\right)^{p/2}.\]\end{lemma}
\begin{proof}
Let $\tilde{t}$ be the vector $(1,\ldots,1,t_{i+1},\ldots,t_d)$, that is the vector $t$ such that all the coordinates lower than $1$ have been replaced by $1$. Let us observe the simple bound that follows from the invariance by translation
\begin{align*}
\int_{[0,t]^p} |F_p(x_1,\ldots,x_p)|\dd x_1,\ldots\dd x_p&\leq \left(\prod_{j=1}^d t_j\right)\int_{[-t,t]^{p-1}} |F_p(x_2,\ldots,x_p)|\dd x_2 \ldots\dd x_p\\
&\leq \left(\prod_{j=1}^d t_j\right)\int_{[-\tilde{t},\tilde{t}]^{p-1}} |F_p(0,x_2,\ldots,x_p)|\dd x_2\ldots\dd x_p\\
&\leq C\left(\prod_{j=1}^it_j\right)\int_{[-2\tilde{t},2\tilde{t}]^p} |F_p(x_1,\ldots,x_p)|\dd x_1\ldots\dd x_p\\
\end{align*}
To estimate the right hand side, it suffices the results in \cite[Sec. 7.3]{Anc24} yield a constant $C$ depending only on the function $g$ and the exponent $p$ such that
\[\int_{[-2\tilde{t},2\tilde{t}]^p} |F_p(x_1,\ldots,x_p)|\dd x_1\ldots\dd x_p\leq C\left(\prod_{j=i+1}^d t_j\right)^{p/2}.\]
\end{proof}
We are now ready to state the following fourth moment estimate.
\begin{proposition}
\label{prop2} Let the assumptions of Lemma \ref{lem1} prevail. Then there is a constant $C$ depending only on the process $f$ such that
\[\E[\nut(A)^2\nut(B)^2]\leq C\left(\prod_{j=1}^i t_j\right)\left(\prod_{j=i+1}^d t_j\right)^2.\]
\end{proposition}
\begin{proof}
Using the expression of the left hand side given by Lemma \ref{lem1} and Lemma \ref{lem2}, and the estimate given by Lemma \ref{lem4}, one obtains the bound
\begin{align*}
\E[\nut(A)^2\nut(B)^2]&\leq C\left(\prod_{j=1}^i t_j\right)\left[\sum_{p=2}^4 \left(\prod_{j=i+1}^d t_j\right)^{p/2} + \left(\prod_{j=i+1}^d t_j\right)^2\right]\\
&\leq C'\left(\prod_{j=1}^i t_j\right)\left(\prod_{j=i+1}^d t_j\right)^2
\end{align*}
\end{proof}
\subsection{Proof of Theorem \ref{thm2}}
\begin{proof}
Let $R\geq 1$ and assume that the two rectangles $A$, $B$ are subsets of the cube $[0,R]^d$. We first use the Cauchy--Schwarz bound and the two moment bounds given by Proposition \ref{prop1} and Proposition \ref{prop2} to get
\begin{align*}
\E[|\nut(A)|^{3/2}|\nut(B)|^{3/2}]&\leq \sqrt{\E[|\nut(A)\nut(B)|]}\sqrt{\E[\nut(A)^2\nut(B)^2]}\\
&\leq C\left(\prod_{j=1}^i t_j\right)\left(\prod_{j=i+1}^d t_j\right)^{3/2}\min\left(1, \left(\prod_{j=1}^i t_j\right)^{\eta/2}\left(\prod_{j=i+1}^d t_j\right)^{1/2}\right)
\end{align*}
If $i=0$, we can bound this term as 
\[\E[|\nut(A)|^{3/2}|\nut(B)|^{3/2}] \leq C\Vol(A\cup B)^{3/2}.\]
If $i\geq 1$, we can use the inequality $\min(1,a)\leq a^\varepsilon$, for $\varepsilon > 0$ to get
\[\E[|\nut(A)|^{3/2}|\nut(B)|^{3/2}]\leq C\Vol(A\cup B)^{1+\frac{\varepsilon \eta}{2}}R^{\frac{1}{2}(d-1)(1 + \varepsilon - \varepsilon\eta)}.\]
Gathering both cases, we get the following bound, which depends only on the volume of $A\cup B$ and the parameter $R$:
\[\E[|\nut(A)|^{3/2}|\nut(B)|^{3/2}] \leq C\Vol(A\cup B)^{3/2} + C\Vol(A\cup B)^{1+\frac{\varepsilon \eta}{2}}R^{\frac{1}{2}(d-1)(1 + \varepsilon - \varepsilon\eta)}.\numberthis\label{eq:05}\]
Now consider the rescaled measure $\nut_R$ on the cube $[0,1]^d$, defined in \eqref{eq:04}. The previous Inequality \eqref{eq:05} then implies that for two adjacent rectangles $A,B$ in $[0,1]^d$, one has
\begin{align*}
\E[|\nut_R(A)|^{3/2}|\nut_R(B)|^{3/2}]&= \frac{1}{\gamma_2^2R^{3d/2}}\E[|\nut(RA)|^{3/2}|\nut(RB)|^{3/2}]\\
&\leq C\Vol(A\cup B)^{3/2} + C\frac{\Vol(A\cup B)^{1+\varepsilon\eta}}{R^{\frac{1}{2}(1 + \varepsilon - \varepsilon\eta - d\varepsilon)}}\\
&\leq C\Vol(A\cup B)^{1+\frac{\varepsilon\eta}{2}},\numberthis\label{eq:09}
\end{align*}
since $A\cup B\in [0,1]^d$, $R\geq 1$ and $\varepsilon$ can be chosen so that $1 + \varepsilon - \varepsilon\eta - d\varepsilon>0$.
\end{proof}
\begin{remark}
In the case $k<d$ one can get, following Remark \ref{rem1}, the more precise bound
\[\E[|\nut_R(A)|^{3/2}|\nut_R(B)|^{3/2}]\leq C\Vol(A\cup B)^{1+\frac{\eta}{2}}.\numberthis\label{eq:10}\]
The choice of the parameter $\alpha$ in \ref{thm2} follows from \eqref{eq:09} and \eqref{eq:10}, and the definition \eqref{eq:07} of $\eta$.
\end{remark}
\printbibliography
\end{document}